\documentclass[reqno]{amsart}

\usepackage{amssymb}
\usepackage{cite}

\AtBeginDocument{{\noindent\small
This is a preprint of a paper whose final and definite form will be published 
in\\ Journal of Mathematical Analysis, ISSN: 2217-3412, Volume 7, Issue 1 (2016).}
\vspace{9mm}}

\begin{document}

\title[Existence of solution to a nonlinear first-order dynamic equation]{%
Existence of solution to a nonlinear first-order dynamic equation on time scales}

\author[B. Bayour, A. Hammoudi, D. F. M. Torres]{Benaoumeur Bayour,
Ahmed Hammoudi, Delfim F. M. Torres}

\address{Benaoumeur Bayour \newline
\indent University of Chlef, B. P. 151, Hay Es-salem Chlef, Algeria}
\email{b.benaoumeur@gmail.com}

\address{Ahmed Hammoudi \newline
\indent Laboratoire de Math\'{e}matiques, Universit\'{e} de Ain T\'{e}mouchentl\newline
\indent B. P. 89, 46000 Ain T\'{e}mouchent, Algeria}
\email{hymmed@hotmail.com}

\address{Delfim F. M. Torres \newline
\indent Center for Research and Development in Mathematics and Applications (CIDMA)\newline
\indent Department of Mathematics, University of Aveiro, 3810-193 Aveiro, Portugal}
\email{delfim@ua.pt}

% --------------------------------------------------------------------------

\thanks{Submitted Jun 10, 2015. Revised and Accepted Dec 02, 2015.}



% --------------------------------------------------------------------------

\begin{abstract}
We prove existence of solution to a nonlinear first-order
nabla dynamic equation on an arbitrary bounded time scale
with boundary conditions, where the right-hand side
of the dynamic equation is a continuous function.
\end{abstract}

\maketitle

% --------------------------------------------------------------------------

\numberwithin{equation}{section}
\newtheorem{theorem}{Theorem}[section]
\newtheorem{lemma}[theorem]{Lemma}
\newtheorem{proposition}[theorem]{Proposition}
\newtheorem{corollary}[theorem]{Corollary}
\newtheorem{remark}[theorem]{Remark}
\newtheorem{definition}[theorem]{Definition}
\newtheorem{example}[theorem]{Example}

% --------------------------------------------------------------------------

\section{Introduction}

In this work we prove existence of solution to the following system:
\begin{equation}
\label{1.1}
\begin{gathered}
x^{\nabla}(t)=f(t,x(t)),
\quad t \in \mathbb{T}_{k},\\
x(a)=x(b).
\end{gathered}
\end{equation}
Here $\mathbb{T}$ is an arbitrary bounded time scale, where we denote
$a:=\min\mathbb{T}$, $b:=\max\mathbb{T}$, $\mathbb{T}_{\circ}=\mathbb{T}\setminus\{a\}$,
and $f:\mathbb{T}_{\circ}\times \mathbb{R}^{n}\rightarrow\mathbb{R}^{n}$
is a continuous function. Problem \eqref{1.1} unifies continuous
and discrete problems. We use the notion of tube solution for system \eqref{1.1},
in the spirit of the works of Gilbert and Frigon \cite{[16],[11],[3]}.
This notion is useful to get existence results for systems of differential
equations of first order, as a generalization of lower and upper solutions
\cite{[15],[13],[12],[14]}. Our main result provides existence of solution
to the nonlinear nabla boundary value problem \eqref{1.1}.

The article is organized as follows. In Section~\ref{sec:2} we review
some basic definitions and theorems regarding $\nabla$-differentiation
and $\nabla$-integration on time scales, and we prove some preliminary results.
In Section~\ref{sec:3} we introduce the notion of tube solution for system
\eqref{1.1} and we prove our main result (Theorem~\ref{thm:mr:exist}).
We end with Section~\ref{sec:conc}, mentioning some directions for future work.

% -----------------------------------------

\section{Preliminaries}
\label{sec:2}

A time scale $\mathbb{T}$ is defined to be any nonempty closed subset
of $\mathbb{R}$. Then the forward and backward jump operators $\sigma,\rho:
\mathbb{T}\rightarrow\mathbb{T}$ are defined by
$$
\sigma(t)=\inf\{s\in\mathbb{T}:s>t\}
\quad \text{ and }
\quad \rho(t)=\sup\{t\in\mathbb{T}:s<t\}.
$$
For $t\in\mathbb{T}$, we say that $t$ is left-scattered
(respectively right-scattered) if $\rho(t)<t$ (respectively $\sigma(t)>t$);
that $t$ is isolated  if it is left-scattered and right-scattered. Similarly,
if $t>\inf(\mathbb{T})$ and $\rho(t)=t$, then we say that $t$ is left-dense;
if $t<\sup(\mathbb{T})$ and $\sigma(t)=t$, then we say that $t$ is right-dense.
Points that are simultaneously left- and right-dense are called dense.
If $\mathbb{T}$ has a right-scattered minimum $m$, then we define
$\mathbb{T}_{\kappa}:=\mathbb{T}-\{m\}$; otherwise, we set
$\mathbb{T}_{\kappa}:=\mathbb{T}$. The (backward) graininess
$\nu:\mathbb{T}_{\kappa}\rightarrow[0,+\infty[$
is defined by $\nu(t):=t-\rho(t)$.

\begin{definition}[See \cite{[2],[3]}]
For $f:\mathbb{T}\rightarrow\mathbb{R}^{n}$ and $t\in\mathbb{T}_{\kappa}$,
the nabla derivative of $f$ at $t$, denoted by $f^{\nabla}(t)$, is defined
to be the number (provided it exists) with the property that given any
$\epsilon>0$ there is a  neighborhood $U$ of $t$ such that
$$
\left\|f(\rho(t))-f(s)-f^{\nabla}(t)[\rho(t)-s]\right\|
\leq \epsilon |\rho(t)-s|
$$
for all $s\in U$. If $f$ is $\nabla$-differentiable at $t$ for every
$t\in\mathbb{T}_{\kappa}$, then $f:\mathbb{T}\rightarrow\mathbb{R}^{n}$
is called the $\nabla$-derivative of $f$ on $\mathbb{T}_{\kappa}$.
\end{definition}

\begin{theorem}[See \cite{[4]}]
Assume $f:\mathbb{T}\rightarrow\mathbb{R}^{n}$
and let $t\in\mathbb{T}_{\kappa}$. The following holds:
\begin{enumerate}
\item If $f$ is $\nabla$-differentiable at $t$, then $f$ is continuous at $t$.

\item If $f$ is continuous at the left-scattered point $t$, then $f$
is $\nabla$-differentiable at $t$ with
$$
f^{\nabla}(t)=\frac{f(t)-f(\rho(t)}{\nu(t)}.
$$

\item If $t$ is left-dense, then $f$ is nabla differentiable at $t$
if and only if the limit
$$
\lim_{s\rightarrow t}\frac{f(t)-f(s)}{t-s}
$$
exists in $\mathbb{R}^{n}$. In this case,
$$
f^{\nabla}(t)=\lim_{s\rightarrow t}\frac{f(t)-f(s)}{t-s}.
$$

\item If $f$ is nabla differentiable at $t$, then
$$
f^{\rho}(t)=f(t)-\nu(t)f^{\nabla}(t),
$$
where $f^{\rho}(t) := f\left(\rho(t)\right)$.
\end{enumerate}
\end{theorem}

\begin{theorem}[See \cite{[4]}]
\label{2}
Assume $f,g:\mathbb{T}\rightarrow\mathbb{R}$ are nabla differentiable at
$t\in\mathbb{T}_{\kappa}$. Then,
\begin{enumerate}
\item $f+g$ is nabla differentiable at $t$ and
$(f+g)^{\nabla}(t)=f^{\nabla}(t)+g^{\nabla}(t)$;

\item $\alpha f$ is nabla differentiable at $t$ for every
$\alpha\in\mathbb{R}$ and $(\alpha f)^{\nabla}(t)=\alpha f^{\nabla}(t)$;

\item $fg$ is nabla differentiable at $t$ and
$$
(fg)^{\nabla}(t)=f^{\nabla}(t)g(t)+f^{\rho}(t)g^{\nabla}(t)
=f(t)g^{\nabla}(t)+f^{\nabla}(t)g^{\rho}(t);
$$

\item if $g(t)g^{\rho}(t)\neq0$, then $\frac{f}{g}$
is nabla differentiable at $t$ and
$$
\left(\frac{f}{g}\right)^{\nabla}(t)
=\frac{f^{\nabla}(t)g(t)-f(t)g^{\nabla}(t)}{g(t)g^{\rho}(t)};
$$

\item if $f$ and $f^{\nabla}$ are continuous, then
$$
\left(\int^{t}_{a}f(t,s)\nabla s\right)^{\nabla}
=f(\rho(t),t)+ \int^{t}_{a}f^{\nabla}(t,s)\nabla s.
$$
\end{enumerate}
\end{theorem}

\begin{theorem}
\label{thm:pt}
Let $W$ be an open set of $\mathbb{R}^{n}$ and $t\in\mathbb{T}$
be a left-dense point. If $g:\mathbb{T}\rightarrow\mathbb{R}^{n}$
is nabla-differentiable at $t$ and $f:W\rightarrow \mathbb{R}$
is differentiable at $g(t)\in W$, then $f \circ g$
is nabla-differentiable at $t$ with
$\left(f \circ g\right)^{\nabla}(t) = \langle f'(g(t)),g^{\nabla}(t)\rangle$.
\end{theorem}

\begin{proof}
Let $\epsilon > 0$. We need to show that there exists
a neighborhood $U$ of $t$ such that
$\left|f(g(t))-f(g(s))-\langle f'(g(t)),g^{\nabla}(t)\rangle (t-s)\right|
\leq \epsilon |t-s|$
for all $s\in U$. Let $k>0$ be a constant and $\epsilon'=\frac{\epsilon}{k}$.
By hypotheses, there exists a neighborhood $U_{1}$ of $t$ where
$\left\| g(t)-g(s)-g^{\nabla}(t)(t-s)\right\| \leq \epsilon' |t-s|$
for all $s\in U_{1}$. In addition, there exists a neighborhood
$V\subset W$ of $g(t)$ such that
$|f(g(t))-f(y)-\langle f'(g(t)),g(t)-y\rangle |
\leq\epsilon'|g(t)-y|$
for all $y\in V$. Since function $g$ is $\nabla$-differentiable at $t$,
it is also continuous at this point, and there exists a neighborhood $U_{2}$
of $t$ such that $g(s)\in V$ for all $s\in U_{2}$.
Let $U:=U_{1}\cap U_{2}$. In this case $U$ is a neighborhood of $t$ and
if $s \in U$, then
\begin{equation*}
\begin{split}
|f&(g(t))-f(g(s))-\langle f'(g(t)),g^{\nabla}(t)\rangle (t-s) |\\
&\leq \left|f(g(t))-f(g(s))-\langle f'(g(t)),g(t)-g(s)\rangle\right|\\
&\qquad + \left|\langle f'(g(t)),g(t)-g(s)\rangle
- \langle f'(g(t)),g^{\nabla}(t)\rangle (t-s)\right|\\
&\leq\epsilon'\left\|g(t)-g(s)\right\|
+\left|\langle f'(g(t)),g(t)-g(s)-g^{\nabla}(t)(t-s)\rangle \right|  \\
&\leq \epsilon' \left(\epsilon'|t-s|+\|g^{\nabla}(t)(t-s)\|\right)
+\|f'(g(t))\| \, \|g(t)-g(s)-g^{\nabla}(t)(t-s)\| \\
&\leq \epsilon'(1+\|g^{\nabla}(t)\|+\|f'(g(t))\|)|t-s|.
\end{split}
\end{equation*}
Put $k=1+\|g^{\nabla}(t)\|+\|f'(g(t))\|$ and the theorem is proved.
\end{proof}

\begin{example}
\label{example}
Assume $x:\mathbb{T}\rightarrow\mathbb{R}^{n}$ is nabla differentiable
at $t\in\mathbb{T}$. We know that $\|\cdot\|:\mathbb{R}^{n}\backslash\{0\}
\rightarrow[0,+\infty[$ is differentiable if $t=\rho(t)$.
It follows from Theorem~\ref{thm:pt} that
$$
\|x(t)\|^{\nabla}=\frac{\langle x(t),x^{\nabla}(t)\rangle}{\|x(t)\|}.
$$
\end{example}

\begin{definition}
A function  $f:\mathbb{T}\rightarrow\mathbb{R}^{n}$
is called ld-continuous provided it is continuous at left-dense points in
$\mathbb{T}$ and its right-sided limits exist (finite) at right-dense points
in $\mathbb{T}$. The set of all ld-continuous functions $f:\mathbb{T}
\rightarrow\mathbb{R}^{n}$ is denoted by $C_{ld}(\mathbb{T},\mathbb{R}^{n})$.
The set of functions $f:\mathbb{T}\rightarrow\mathbb{R}^{n}$ that are
nabla-differentiable and whose nabla-derivative is ld-continuous, is denoted
by $C_{ld}^{1}(\mathbb{T},\mathbb{R}^{n})$. It is known that if $f$ is
ld-continuous, then there is a function $F$ such that $F^{\nabla}=f$
\cite{[5]}. In this case,
$$
\int_{a}^{b}f(t)\nabla t := F(b)-F(a).
$$
\end{definition}

\begin{theorem}[See \cite{[1]}]
Assume $a,b,c\in\mathbb{T}$. Then
\begin{enumerate}
\item $\int_{a}^{b}[f(t)+g(t)]\nabla t
=\int_{a}^{b}f(t)\nabla t+\int_{a}^{b}g(t)\nabla t$;

\item $\int_{a}^{b}kf(t)\nabla t=k\int_{a}^{b}f(t)\nabla t$;

\item $\int_{a}^{b}f(t)\nabla t=-\int_{b}^{a}f(t)\nabla t$;

\item $\int_{a}^{b}f(t)\nabla t=\int_{a}^{c}f(t)\nabla t
+\int_{c}^{b}f(t)\nabla t$;

\item $\int_{a}^{b}f^{\nabla}(t)g(t)\nabla t= \left.f(t)g(t)\right|_{a}^{b}
-\int_{a}^{b}f^{\rho}(t)g^{\nabla}(t)\nabla t$.
\end{enumerate}
\end{theorem}

\begin{theorem}[See \cite{[1]}]
\label{5}
The following inequalities hold:
$$
\left|\int_{a}^{b}f(t)g(t)\nabla t\right|
\leq \int_{a}^{b}|f(t)g(t)|\nabla t
\leq \left( \max_{\sigma(a)\leq t\leq b}|f(t)|\right)
\int_{a}^{b} |g(t)|\nabla t.
$$
\end{theorem}

\begin{definition}[See \cite{[4]}]
For $\epsilon>0$, the (nabla) exponential function $\hat{e}_{\epsilon}(\cdot,t_{0})
:\mathbb{T}\rightarrow \mathbb{R}$ is defined as the unique
solution to the initial value problem
$$
x^{\nabla}(t)=\epsilon x(t),
\quad x(t_{0})=1.
$$
More explicitly, the exponential function $\hat{e}_{\epsilon}(\cdot,t_{0})
:\mathbb{T}\rightarrow \mathbb{R}$ is given by the formula
$$
\hat{e}_{\epsilon}(t,t_{0})
=exp\left(\int_{t_{0}}^{t}\hat{\xi}_{\epsilon}(\nu(s))\nabla s\right),
$$
where for $h\geq 0$ we define $\hat{\xi}_{\epsilon}(h)$ as
\begin{displaymath}
\hat{\xi}_{\epsilon}(h)
= \left\{
\begin{array}{ll}
\epsilon & \textrm{if $h=0$},\\
-\frac{log(1-h\epsilon)}{h} & \textrm{otherwise}.
\end{array} \right.
\end{displaymath}
\end{definition}

\begin{proposition}
\label{***}
If $g\in C^{1}(\mathbb{T}_{\kappa},\mathbb{R}^{n})$, then function
$x:\mathbb{T}\rightarrow \mathbb{R}^{n}$ defined by
$$
x(t) =\hat{e}_{1}(t,b)\left[\frac{\hat{e}_{1}(a,b)}{\hat{e}_{1}(a,b)-1}
\int_{(a,b]\cap \mathbb{T}}\frac{g(s)}{\hat{e}_{1}(\rho(s),b)}\nabla s
-\int_{(t,b]\cap \mathbb{T}}\frac{g(s)}{\hat{e}_{1}(\rho(s),b)}\nabla s\right]
$$
is solution to the problem
\begin{equation}
\label{1.4}
\begin{gathered}
x^{\nabla}(t)-x(t)=g(t),
\quad t\in\mathbb{T}_{\kappa},\\
x(a)=x(b).
\end{gathered}
\end{equation}
\end{proposition}

\begin{proof}
We check \eqref{1.4} for each pair $(x_{i},g_{i})$, $i\in\{1,2,\ldots,n\}$,
by direct calculation. To simplify notation, we omit the indices $i$
and we write
$$
k=\frac{\hat{e}_{1}(a,b)}{\hat{e}_{1}(a,b)-1}
\int_{(a,b]\cap \mathbb{T}}\frac{g(s)}{\hat{e}_{1}(\rho(s),b)}\nabla s.
$$
From Theorem~\ref{2}, we have that
\begin{multline*}
x^{\nabla}(t)-x(t)
=\hat{e}_{1}(t,b)k-\hat{e}_{1}(t,b)\int_{(a,b]\cap \mathbb{T}}\frac{g(s)}{\hat{e}_{1}(\rho(s),b)}\nabla s\\
+\hat{e}_{1}(\rho(t),b)\frac{g(t)}{\hat{e}_{1}(\rho(t),b)}-\hat{e}_{1}(t,b)k
+\hat{e}_{1}(t,b)\int_{(a,b]\cap \mathbb{T}}\frac{g(s)}{\hat{e}_{1}(\rho(s),b)}\nabla s
= g(t)
\end{multline*}
for all $t\in \mathbb{T}_{\kappa}$. It is easy to verify that $x (a) = x (b)$.
\end{proof}

\begin{lemma}
\label{lem}
Let $r\in C_{ld}^{1}(\mathbb{T},\mathbb{R}^{n})$ be a function such that
$r^{\nabla}(t)<0$ for all $t\in\{t\in\mathbb{T}_{\kappa};r(t)>0\}$.
If $r(a)\geq r(b)$, then $r(t) \leq 0$ for all $t\in\mathbb{T}$.
\end{lemma}

\begin{proof}
Suppose that there exists a $t \in \mathbb{T}$ such that $r(t)> 0$.
Then there exists a $t_{0}\in\mathbb{T}$ such that
$r(t_{0})=\max_{t\in\mathbb{T}}(r(t)>0)$. If $\rho(t_{0})<t_{0}$, then
$$
r^{\nabla}(t_{0})=\frac{r(\rho(t_{0}))-r(t_{0}}{\rho(t_{0})-t_{0}}\geq 0,
$$
which contradicts the hypothesis. If $t_{0}>a$ and $t_{0}=\rho(t_{0})$, then
there exists an interval $[t_{1},t_{0}]$ such that $r(t)>0$
for all $t\in[t_{1},t_{0}]$. Thus
$$
\int_{t_{1}}^{t_{0}}r^{\nabla}(s)\nabla s=r(t_{0})-r(t_{1})<0,
$$
which contradicts the maximality of $r(t_{0})$. Finally, if $t_{0}=a$, then
by hypothesis $r(b)\geq r(a)$ gives $r (a) = r (b)$. Taking $t_{0}= a$,
one can check that $r(a) \leq 0$ by using previous steps of the proof.
The lemma is proved.
\end{proof}

% -----------------------------------------

\section{Main Result}
\label{sec:3}

In this section we prove existence of solution to
problem \eqref{1.1}. A solution of this
problem is a function $x\in C^{1}_{ld}(\mathbb{T},\mathbb{R}^{n})$
satisfying \eqref{1.1}. Let us recall that $\mathbb{T}$ is bounded
with $a=\min\mathbb{T}$ and $b=\max\mathbb{T}$. We introduce the notion
of tube solution for problem \eqref{1.1} as follows.

\begin{definition}
Let $(v,M)\in C^{1}_{ld}(\mathbb{T},\mathbb{R}^{n})
\times C^{1}_{ld}(\mathbb{T},[0,+\infty[)$. We say that $(v,M)$
is a tube solution of \eqref{1.1} if
\begin{enumerate}
\item $\langle x-v(t),f(t,x(t))-v^{\nabla}(t)\rangle
+M(t)\|x-v(t)\|\leq M(t)M^{\nabla}(t)$
for every $t\in \mathbb{T}_{\kappa}$ and for every
$x\in\mathbb{R}^{n} $ such that $\|x-v(t)\|= M(t)$;

\item $v^{\nabla}(t)=f(t,v(t))$ and $\|x-v(t)\|-M^{\nabla}(t)<0$
for every $t\in \mathbb{T}_{\kappa}$ such that $M(t)=0$;

\item $\|v(a)- v(b)\|\leq M(a)-M(b)$.
\end{enumerate}
\end{definition}

Let $\mathbf{T}(v,M) := \{x\in C^{1}_{ld}(\mathbb{T},\mathbb{R}^{n})
:\|x(t)-v(t)\| \leq M(t)$ for every $t\in \mathbb{T}\}$.
We consider the following problem:
\begin{equation}
\label{1.3}
\begin{gathered}
x^{\nabla}(t)-x(t)=f(t,\widehat{x}(t))-\widehat{x}(t),
\quad t\in \mathbb{T}_{\kappa},\\
x(a)=x(b),
\end{gathered}
\end{equation}
where
\begin{displaymath}
\hat{x}(t)
= \left\{
\begin{array}{ll}
\frac{M(t)}{\|x-v(t)\|}(x(t)-v(t))+v(t)& \textrm{if $\|x-v(t)\|> M(t)$},\\
x(t) & \textrm{otherwise}.
\end{array} \right.
\end{displaymath}
Let us define the operator $\mathbf{T}_{\hat{p}}: C(\mathbb{T},\mathbb{R}^{n})
\rightarrow C(\mathbb{T},\mathbb{R}^{n})$ by
\begin{multline*}
\mathbf{T}_{\hat{p}}(x)(t)=\hat{e}_{1}(t,b)\left[
\frac{\hat{e}_{1}(a,b)}{\hat{e}_{1}(a,b)-1}\int_{(a,b]\cap \mathbb{T}}
\dfrac{f(s,\hat{x}(s))-\hat{x}(s)}{\hat{e}_{1}(\rho(s),b)}\nabla s\right.\\
\left.-\int_{(t,b]\cap\mathbb{T}}\dfrac{f(s,\hat{x}(s))
-\hat{x}(s)}{\hat{e}_{1}(\rho(s),b)}\nabla s \right].
\end{multline*}

\begin{proposition}
\label{pt}
If $(v,M)\in C^{1}_{ld}(\mathbb{T},\mathbb{R}^{n})
\times C^{1}_{ld}(\mathbb{T},[0,+\infty[)$ is a tube solution of
\eqref{1.1}, then $\mathbf{T}_{\hat{p}} : C(\mathbb{T},\mathbb{R}^{n})
\rightarrow C(\mathbb{T},\mathbb{R}^{n})$ is compact.
\end{proposition}

\begin{proof}
We first prove the continuity of the operator $\mathbf{T}_{\hat{p}}$.
Let $\{x_{n}\}_{n\in \mathbb{N}}$ be a sequence of $C(\mathbb{T},\mathbb{R}^{n})$
converging to $x\in C(\mathbb{T},\mathbb{R}^{n})$. By Theorem~\ref{5},
\begin{equation*}
\begin{split}
\| \mathbf{T}_{\hat{p}}&(x_{n})(t)-\mathbf{T}_{\hat{p}}(x)(t)\|\\
&\leq (1+c)\|\hat{e}_{1}(t,b)\|
\, \left\|\int_{(a,b]\cap \mathbb{T}}\dfrac{f(s,\hat{x}_{n}(s))
-f(s,\hat{x}(s))-(\hat{x}_{n}(s)-\hat{x}(s))}{\hat{e}_{1}(\rho(s),b)}\nabla s \right\|\\
&\leq \frac{k(1+c)}{M}\left(\int_{(a,b]
\cap \mathbb{T}}\left\|f(s,\hat{x}_{n}(s))-f(s,\hat{x}(s))\right\|
+\left\|\hat{x}_{n}(s)-\hat{x}(s)\right\|\nabla s\right),
\end{split}
\end{equation*}
where $k:=\max_{t\in\mathbf{T}}|\hat{e}_{1}(t,b)|$,
$M:=\min_{t\in\mathbf{T}}(\hat{e}_{1}(t,b))$, and
$c:=\|\frac{\hat{e}_{1}(a,b)}{\hat{e}_{1}(a,b)-1}\|$.
Since there is a constant $R>0$ such that $\|\hat{x}\|_{C(\mathbb{T},\mathbb{R}^{n})}<R$,
there exists an index $N$ such that $\|\hat{x}_{n}\|_{C(\mathbb{T},\mathbb{R}^{n})}<R$
for all $n>N$. Thus $f$ is uniformly continuous on $\mathbb{T}_{\kappa}\times B_{R}(0)$.
Therefore, for $\epsilon>0$ given, there is a $\delta>0$ such that
for all $x,y\in \mathbb{R}^{n}$, where
\begin{equation*}
\|x-y\|<\delta<\frac{\epsilon M}{2k(1+c)(b-a)},
\end{equation*}
one has
\begin{equation*}
\|f(s,y)-f(s,x)\|<\frac{\epsilon M}{2k(1+c)(b-a)}.
\end{equation*}
By assumption, for all $s\in\mathbb{T}_{\kappa}$ it is possible to find
an index $\hat{N}>N$ such that
$\|\hat{x}_{n}-\hat{x}\|_{C(\mathbb{T},\mathbb{R}^{n})}<\delta$
for $n>\hat{N}$. In this case,
\begin{equation*}
\| \mathbf{T}_{\hat{p}}(x_{n})(t)-\mathbf{T}_{\hat{p}}(x)(t)\|
\leq \frac{2k(1+c)}{M}\int_{[a,b)\cap \mathbb{T}}\frac{\epsilon M}{2k(1+c)(b-a)}\nabla s
\leq \epsilon.
\end{equation*}
This proves the continuity of $\mathbf{T}_{\hat{p}}$.
We now show that the set $\mathbf{T}_{\hat{p}}(C(\mathbb{T},\mathbb{R}^{n}))$
is relatively compact. Consider a sequence $\{y_{n}\}_{n\in\mathbb{N}}$ of
$\mathbf{T}_{\hat{p}}(C(\mathbb{T},\mathbb{R}^{n}))$ for all $n\in\mathbb{N}$.
It exists $x_{n}\in C(\mathbb{T},\mathbb{R}^{n})$ such that
$y_{n}=\mathbf{T}_{\hat{p}}(x_{n})$. From Theorem~\ref{5} one has
$$
\|\mathbf{T}_{\hat{p}}(x_{n})(t)\|\leq\frac{k(1+c)}{M} \left(\int_{[a,b)
\cap \mathbb{T}}\|f(s,\hat{x}_{n}(s)) \|\nabla s
+\int_{[a,b)\cap \mathbb{T}}\|\hat{x}_{n}(s)\|\nabla s\right).
$$
By definition, there is an $R>0$ such that $\|\hat{x}_{n}(s)\|\leq R$ for all
$s\in\mathbb{T}$ and all $n\in\mathbb{N}$. Function $f$ is compact on
$\mathbb{T}_{\kappa}\times B_{R}(0)$ and we deduce the existence of a constant
$A>0$ such that $\|f(s,\hat{x}_{n}(s)\|\leq A$ for all $s\in\mathbb{T}_{\kappa}$
and all $n\in\mathbb{N}$. The sequence $\{y_{n}\}_{n}\in\mathbb{N}$
is uniformly bounded. Note also that
\begin{multline*}
\| \mathbf{T}_{\hat{p}}(x_{n})(t_{2})-\mathbf{T}_{\hat{p}}(x_{n})(t_{1})\|
\leq B\|\hat{e}_{1}(t_{2},b)- \hat{e}_{1}(t_{1},b)\|\\
+k\left\|\int_{(a,b]\cap \mathbb{T}} \frac{f(s,\hat{x}_{n}(s))
-\hat{x}_{n}(s)}{\hat{e}_{1}(\rho(s),b)}\nabla s \right\|
< B\|\hat{e}_{1}(t_{2},b)- \hat{e}_{1}(t_{1},b)\|+\frac{k(A+R)}{M}|t_{2}-t_{1}|
\end{multline*}
for $t_{1},t_{2}\in\mathbb{T}$, where $B$ is a constant
that can be chosen such that it is higher than
$$
\sup_{n\in\mathbb{N}}\left\|\frac{\hat{e}_{1}(a,b)}{\hat{e}_{1}(a,b)-1}
\int_{(a,b]\cap \mathbb{T}}\dfrac{f(s,\hat{x}_{n}(s))
-\hat{x}_{n}(s)}{\hat{e}_{1}(\rho(s),b)}\nabla s
+\int_{(t,b]\cap\mathbb{T}}\dfrac{f(s,\hat{x}_{n}(s))
-\hat{x}_{n}(s)}{\hat{e}_{1}(\rho(s),b)}\nabla s\right\|.
$$
This proves that the sequence $\{y_{n}\}_{n\in \mathbb{N}}$ is equicontinuous.
It follows from the Arzel\`{a}--Ascoli theorem, adapted to our context, that
$\mathbf{T}_{\hat{p}}(C(\mathbb{T},\mathbb{R}^{n}))$
is relatively compact. Hence $\mathbf{T}_{\hat{p}}$ is compact.
\end{proof}

\begin{theorem}
\label{thm:mr:exist}
If $(v,M)\in C^{1}_{ld}(\mathbb{T},[0,+\infty[)
\times C^{1}_{ld}(\mathbb{T},\mathbb{R}^{n})$ is a tube solution of
\eqref{1.1}, then problem \eqref{1.1} has a solution
$x\in C^{1}_{ld}(\mathbb{T},\mathbb{R}^{n})\cap \mathbf{T}(v,M)$.
\end{theorem}

\begin{proof}
By Proposition~\ref{pt}, $\mathbf{T}_{\hat{p}}$ is compact. It has a fixed point
by Schauder's fixed point theorem. Proposition~\ref{***} implies that this fixed
point is a solution to problem \eqref{1.3}. Then it suffices to show that
for every solution $x$ of \eqref{1.3} one has $x\in\mathbf{T}(v,M)$.
Consider the set $A=\{t\in\mathbb{T}_{\kappa}:\|x(t)-v(t)\|>M(t)\}$.
If $t\in A$ is left dense, then by virtue of Example~\ref{example} we have
$$
\left(\parallel x(t)-v(t)\parallel-M(t)\right)^{\nabla}
= \frac{\langle x(t)-v(t),x^{\nabla}(t)
-v^{\nabla}(t)\rangle}{\| x(t)-v(t)\|}-M^{\nabla}(t).
$$
If $t\in A$ is left scattered, then
\begin{equation*}
\begin{split}
(\| x(t)-&v(t)\|-M(t))^{\nabla}
=\| x(t)-v(t)\|^{\nabla}-M^{\nabla}(t)\\
&= \frac{\| x(t)-v(t)\|^{2}-\| x(t)-v(t)\|\| x(\rho(t))
-v(\rho(t))\|}{\nu(t)\| x(t)-v(t)\|}-M^{\nabla}(t) \\
&\leq\frac{\langle x(t)-v(t),x(t)-v(t)- x(\rho(t))
+v(\rho(t))\rangle}{\nu(t)\| x(t)-v(t)\|}-M^{\nabla}(t ) \\
&= \frac{\langle x(t)-v(t),[f(t,\hat{x}(t))-\hat{x}(t)+x(t)]
-v^{\nabla}(t)\rangle}{\| x(t)-v(t)\|}-M^{\nabla}(t).
\end{split}
\end{equation*}
We will show that if $t\in A$, then $\left(\| x(t)-v(t)\|-M(t)\right)^{\nabla}<0$.
If $t\in A$ and $M(t)>0$, then
\begin{equation*}
\begin{split}
(\| x(t)-&v(t)\|-M(t))^{\nabla}=\| x(t)-v(t)\|^{\nabla} -M^{\nabla}(t)\\
&= \frac{\| x(t)-v(t)\|^{2}-\| x(t)-v(t)\|\| x(\rho(t))
-v(\rho(t))\|}{\nu(t)\| x(t)-v(t)\|}-M^{\nabla}(t) \\
&\leq \frac{\langle x(t)-v(t),x(t)-v(t)- x(\rho(t))
+v(\rho(t))\rangle}{\nu(t)\| x(t)-v(t)\|}-M^{\nabla}(t) \\
&=\frac{\langle x(t)-v(t),x^{\nabla}(t)
-v^{\nabla}(t)\rangle}{\| x(t)-v(t)\|}-M^{\nabla}(t)\\
&=\frac{\langle x(t)-v(t),f(t,\hat{x}(t))-v^{\nabla}(t)\rangle}{\| x(t)-v(t)\|}
+\frac{\langle x(t)-v(t),-\hat{x}(t)+x(t)\rangle}{\|x(t)-v(t)\|}-M^{\nabla}(t)  \\
&= \frac{\langle \hat{x}^{\nabla}(t)-v(t),f(t,\hat{x}(t))-v^{\nabla}(t)\rangle}{M(t)}
-M(t)+\| x(t)-v(t)\| -M^{\nabla}(t)\\
&\leq\frac{M(t)M^{\nabla}(t)-M(t)\| x(t)-v(t)\|}{M(t)}-M(t)
+\| x(t)-v(t)\|-M^{\nabla}(t)\\
&= -M(t)<0.
\end{split}
\end{equation*}
In addition, if $M(t)=0$, then
\begin{equation*}
\begin{split}
(\| x(t)-&v(t)\|-M(t))^{\nabla}
= \frac{\langle x(t)-v(t),f(t,\hat{x}(t))+[x(t)-\hat{x}(t)]
-v^{\nabla}(t)\rangle}{\| x(t)-v(t)\|}-M^{\nabla}(t) \\
&\leq\frac{\langle x(t)-v(t),f(t,v(t))-v^{\nabla}(t)\rangle}{\| x(t)-v(t)\|}
+\| x(t)-v(t)\|-M^{\nabla}(t)
<0.
\end{split}
\end{equation*}
If we set $r(t):=\| x(t)-v(t)\|-M(t)$, then $r^{\nabla}(t)<0$
for every $t\in\{t\in\mathbb{T}_{\kappa},r(t)\geq 0\}$.
Moreover, since $(v,M)$ is a tube solution of \eqref{1.1}, one has
$$
r(a)-r(b)\leq\| v(a)-v(b)\|-(M(a)-M(b))\leq0
$$
and thus the hypotheses of Lemma~\ref{lem} are satisfied,
which proves the theorem.
\end{proof}

\begin{example}
Consider the following boundary value problem on time scales:
\begin{equation}
\label{eq:prb:ex:exist}
\begin{gathered}
x^{\nabla}(t) = a_{1}\|x(t)\|^{2}x(t)
-a_{2}x(t)+a_{3}\varphi(t),
\quad t\in\mathbb{T}_{\kappa},\\
x(a) = x(b),
\end{gathered}
\end{equation}
where $a_{1}$, $a_{2}$, $a_{3}\geq 0$  are nonnegative real constants chosen
such that $a_{2}\geq a_{1}+a_{3}+1$ and $\varphi:\mathbb{T}_{\kappa}
\rightarrow \mathbb{R}^{n}$ is a continuous function satisfying
$\|\varphi(t)\|=1$ for every $t\in\mathbb{T}_{\kappa}$.
It is easy to check that $(v,m)\equiv(0,1)$ is a tube solution.
By Theorem~\ref{thm:mr:exist}, problem \eqref{eq:prb:ex:exist}
has a solution $x$ such that $\|x(t)\|\leq1$ for every $t\in \mathbb{T}$.
\end{example}

% -----------------------------------------

\section{Conclusion and Future Work}
\label{sec:conc}

We proved existence of a solution to a nonlinear first-order
nabla dynamic equation on time scales. For that the notion
of tube solution is used, in the spirit of the works
of Frigon and Gilbert \cite{[16],[11],[3]}.
Our results can be improved by using $\nabla$-Caratheodory functions
$f$ on the right-hand side of equation \eqref{1.1}, which are not necessarily
continuous. For that one needs to define a proper Sobolev space and
related nabla concepts. This is under investigation and will be addressed elsewhere.

% -----------------------------------------

\subsection*{Acknowledgments}

This research is part of first author's Ph.D. project,
which is carried out at Sidi Bel Abbes University, Algeria.
It was essentially finished while Bayour was visiting the Department
of Mathematics of University of Aveiro, Portugal, February to April of 2015.
The hospitality of the host institution and the financial support
of University of Chlef, Algeria, are here gratefully acknowledged.
Torres was supported by Portuguese funds through CIDMA and FCT,
within project UID/MAT/04106/2013.

% ------------------------------------------

% ------------------------------------------

\end{document}